\documentclass[a4paper,10pt, notitlepage, twoside]{article}
\usepackage[utf8,latin1]{inputenc}
\usepackage{enumerate}
\usepackage{amssymb,amsfonts,amsmath,amsthm}
\usepackage{xfrac}
\usepackage{setspace}
\usepackage{titlesec}
\usepackage{abstract}
\usepackage[inner=3cm,outer=3cm]{geometry}
\usepackage{fancyhdr}
\usepackage{hyperref}

\DeclareMathOperator{\dir}{dir}
\DeclareMathOperator{\res}{res}

\renewcommand{\leftmark}{\textsc{\small Stratifications of Tangent Cones}}
\renewcommand{\rightmark}{\textsc{\small Erick Garcia Ramirez}} 

\fancypagestyle{firststyle}{
\fancyhf{}
\fancyfoot[C]{\thepage}
}

\fancypagestyle{plain}{%
\fancyhead{} 
}

\titleformat
{\section} 
[hang] 
{\bfseries \large} 
{\thesection.} 
{1ex} 
{\centering
\vspace{-.1cm}}

\titleformat
{\subsection} 
[block] 
{\bfseries} 
{\thesubsection. } 
{0ex} 
{}

\newtheorem{Theorem}{Theorem}[section]
\newtheorem{Proposition}[Theorem]{Proposition}
\newtheorem{Lemma}[Theorem]{Lemma}
\newtheorem{Corollary}[Theorem]{Corollary}
\newtheorem{Fact}[Theorem]{Fact}

\theoremstyle{definition}
\newtheorem{Definition}[Theorem]{Definition}

\newtheorem{Example}[Theorem]{Example}
\newtheorem{Remark}[Theorem]{Remark}

\begin{document}
\title{Stratifications of tangent cones in real closed (valued) fields}

\author{Erick Garc\'ia Ram\'irez}

\maketitle
\begin{abstract} 
{\bf Abstract.} We introduce tangent cones of subsets of cartesian powers of a real closed field, generalising the notion of the classical tangent cones of subsets of Euclidean space. We then study the impact of non-archimedean stratifications (t-stratifications) on these  tangent cones. Our main result is that a t-stratification induces stratifications of the same nature on the tangent cones of a definable set. As a consequence, we show that the archimedean counterpart of a t-stratification induces Whitney stratifications on the tangent cones of a semi-algebraic set. The latter statement is achieved by working with the natural valuative structure of non-standard models of the real field. 
\end{abstract}


\begin{section}{Introduction}
 Tangent cones and regular stratifications are employed in the study of the local geometry of sets. Tangent cones generalise the tangent spaces to singular points and stratifications are (usually) partitions with adequate smoothness properties and conditions on how the pieces fit together (regularity conditions). Examples of regular stratifications are Whitney stratifications in the real Euclidean space.
 
  In the present work we are firstly interested in a generalisation of tangent cones. We define the tangent cone of any subset of the cartesian product of a real closed field, which is a straightforward translation of the classical definition of the tangent cone of a subset of $\mathbb R^n$. By focusing on the case of the real closed field being non-archimedean, we introduce a valuation-theoretic setting and we study tangent cones of subsets defined in this valuative structure. Secondly, we are interested in regular stratifications in this valued field setting. I. Halupczok introduced \emph{t-stratifications} in~\cite{immi}, a notion of non-archimedean stratifications in Henselian valued fields of equi-characteristic 0, and this is the notion we work with in this paper. Our aim is to study the impact of these stratifications on tangent cones, adding another application of these novel stratifications (see~\cite{immi} for previous applications).
 
 If $R$ is a non-archimedean real closed field (e.g. a non-principal ultrapower of $\mathbb R$) and $V$ is a proper convex valuation ring of $R$ (e.g. the set of finite numbers in $R$), we regard $R$ as an $\mathcal L_V:=\mathcal L_{or}\cup \{V\}-$structure, where $\mathcal L_{or}$ is the language of ordered rings and $V$ is a unary predicate obviously interpreted. Our main result is the following. 
 
 \begin{Theorem}\label{intro1}
  Let $X$ be an $\mathcal L_V$-definable subset of $R^n$ and $(S_i)_i$ a t-stratification of $X$. Then $(S_i)_i$ induces a t-stratification on the tangent cone $\mathcal C_p(X)$ of $X$ at the point $p$.
 \end{Theorem}
 
 The conclusion means that if we define the sets $\mathcal C_{p,0}:=\mathcal C_p(S_0)$ and $\mathcal C_{p,i+1}:=\mathcal C_{p}(S_0\cup\dots\cup S_{i+1})\setminus \mathcal C_{p}(S_0\cup\dots\cup S_{i})$ for $0\leq i<n$, then the partition $(\mathcal C_{p,i})_i$ is a t-stratification of $\mathcal C_p(X)$. The definition of these specific strata for $\mathcal C_p(X)$ is  natural. 

 For the proof of this theorem we make use of a description of $\mathcal C_p(X)$ via $\mathcal L_V$-definable differentiable curves, available whenever $X$ is $\mathcal L_V$-definable itself (Proposition~\ref{curvedef2}). This description is obtained from results of L. van den Dries and A. Lewenberg in~\cite{driesI} and~\cite{driesII} on real closed fields equipped with a proper convex valuation ring closed under all 0-definable continuous functions (T-convexity).
 
 We also explore the implications of Theorem~\ref{intro1} on classical tangent cones in $\mathbb R$. For this we fix a non-standard model $^*\mathbb R$ of $\mathbb R$ as our non-archimedean real closed field and then we say that a semi-algebraic partition $(S_i)_i$ of $\mathbb R^n$ is an archimedean t-stratification of $X\subseteq \mathbb R^n$ if $(^*S_i)_i$ is a t-stratification of $^*X$. Note that by a result in~\cite{immi}, any archimedean t-stratification is a Whitney stratification. We obtain the following.
 \begin{Theorem}\label{intro2}
  Let $X$ be a semi-algebric subset of $\mathbb R^n$ and $(S_i)_i$ be an archimedean t-stratification of $X$. Then $(S_i)_i$ induces an archimedean t-stratifications on the tangent cone $C_p(X)$ of $X$ at $p$.
 \end{Theorem}

 The meaning of inducing an archimedean t-stratification is parallel to the one of inducing a t-stratification. It then also follows that an archimedean t-stratification induces Whitney stratifications on tangent cones. This contrasts with the case of Whitney stratifications as they are known to not be enough to induce Whitney stratifications on tangent cones (see Example~\ref{example}). 

 The paper is organised as follows. In Section 1 we fix the setting for most of the paper, followed by the definition of tangent cones. In the second part of such section we prove the description of $\mathcal C_p(X)$ via definable curves mentioned earlier. Section 2 starts by introducing t-stratifications and related concepts. Then we prove Proposition~\ref{stronger}, which is a strong statement about the relation of sets and their tangent cones. In Subsection 2.3 we prove Theorem~\ref{intro1}. Section 3 contains our results on classical tangent cones. The paper ends with a discussion on generalisations of our main results to a wider family of definable sets. 
 \end{section}
 
\vspace{0.1cm} 
{\bf Acknowledgements.} The present work has been supported by scholarships from CONACYT and DGRI-SEP. It is part of the author's PhD project at the University of Leeds under the supervision of Dr Immanuel Halupczok and Prof Dugald Macpherson, to whom the author is greatly indebted.

\begin{section}{Tangent cones in real closed (valued) fields}
Before discussing and presenting our definition of tangent cones in real closed fields, we discuss the general setting for most of this paper. By $\mathcal L_{or}$ we denote the language of ordered rings, $(+,-,\cdot,0,1,<)$. A real closed field $R$ is naturally seen as an $\mathcal L_{or}$-structure, but when $R$ is non-archimedean, it is also natural to consider $R$ as a valued field and modify the language accordingly. Specifically, if $R$ is non-archimedean and $V$ is a proper convex valuation ring of $R$ (e.g. the set $\{x\in R\ |\ \exists N\in \mathbb N(-N\leq x\leq N)\}$ of finite numbers in $R$), we regard $R$ as an $\mathcal L_V$-structure, where $\mathcal L_V:=\mathcal L_{or}\cup \{V\}$ and the unary predicate $V$ is interpreted in the obvious way. By a definable set of $R$ we mean an $\mathcal L_V$-definable subset $X$ of $R^n$ (for some $n>0$) where we allow arbitrary parameters from $R$. Let us further mention that as an $\mathcal L_V$-structure, $R$ is weakly o-minimal, i.e. any definable subset of $R$ is a finite union of convex definable sets. 

We fix some more notation. By $\Gamma$ and $\overline{R}$ we denote respectively the value group and the residue field of $R$. Then $\Gamma$ is a divisible group (and is isomorphic to $R^{\times}/\mathcal U(V)$, where $\mathcal U(V)$ is the set of units of $V$) and $\overline R$ is real closed. We let $v:R\longrightarrow \Gamma_{\infty}$ denote the valuation map on $R$ and we consider its multi-dimensional version given by $\hat v(a_1,\dots, a_n):=\min\{v(a_1),\dots, v(a_n)\}$ on $R^n$. The residue map is denoted by $\res:R\longrightarrow \overline{R}$. $R$ has a natural definable norm taking values in $R_{\geq 0}$, and we use $\|\cdot\|_{_R}$ to denote it. That $V$, the valuation ring of $R$, is a convex  set in $R$ implies that the topology defined by the valuation (\emph{the valuative topology}) and the one induce by this norm coincide; we use this fact freely.
 
Now we introduce tangent cones in $R$. Following the classical work of H. Whitney~\cite{whitney}, in $\mathbb C^n$ the tangent cone of a set $X$ at a point $p\in \mathbb C^n$ is defined to be the union of all the \emph{limiting secant lines} to $X$ at $p$. This definition makes full sense for subsets of $\mathbb R^n$ but loses the tight relation with the local geometry of the set. This can be seen in the cusp curve, the set defined by $x^3-y^2=0$ in $\mathbb R^2$. At 0, it has only one limiting secant line, the horizontal axis, so this line would be its tangent cone. Clearly the negative part of the axis conveys little information about the set. In order to recover the tighter relation of the tangent cone with the local geometry of the set, it is customary to define the tangent cone as the union of all the \emph{limiting secant rays} to $X$ at $p$. A \emph{ray} is a set of points of the form $tx$, with fixed unitary $x\in R^n$ (the direction of the ray) and $t$ ranging in $[0,\infty)$. The propriety of such definition is exemplified by the applications of these tangent cones, see e.g.~\cite{wilson} and~\cite{fortuna} on matching prescribed tangent cones to algebraic subsets of $\mathbb R^n$. This is the definition of tangent cones that we generalise for subsets of $R^n$. Notice that this definition applies also in the case of archimedean $R$.
 
 \begin{Definition}\label{tcdef}
Let $X\subseteq R^n$ and $p\in R^n$. We define the \textit{tangent cone of $X$ at} $p$ to be the set 
\begin{align*}
\mathcal C_p(X):=\{y\in R^n\ |\ \text{there exist sequences } (x_\mu)\subseteq X \text{ and } (a_\mu)\subseteq R_{>0}, \\ 
\text{ such that } \lim\limits_{\mu\rightarrow \infty}x_\mu=p \text{ \ and } \lim\limits_{\mu\rightarrow \infty}a_\mu(x_\mu-p)=y\}. 
\end{align*}
 \end{Definition}
\noindent The following are immediate properties of tangent cones. 

\begin{Proposition} Let $X,Y\subseteq R^n$ and $p\in R^n$. Then the following holds.
\begin{enumerate}[(i)]
 \item if $X\subseteq Y$, $\mathcal C_p(X)\subseteq \mathcal C_p(Y)$;
 \item  $\mathcal C_p(X\cup Y)=\mathcal C_p(X)\cup \mathcal C_p(Y)$;
 \item $\mathcal C_p(X\cap Y)\subseteq \mathcal C_p(X)\cap \mathcal C_p(Y)$, and the strict relation may hold;
 \item $\mathcal C_p(X)$ is a closed set; 
 \item if $X$ is definable (with parameters), $\mathcal C_p(X)$ is definable (over the same parameters plus $p$).
\end{enumerate}
\end{Proposition}
\noindent Item \textit{(vii)} is an obvious consequence of the following alternative definition of the tangent cone:
  \begin{align*}
  \mathcal C_p(X)=\{y\in R^n\ |\ \forall \varepsilon \in R_{>0}\exists x\in X\exists a\in R_{>0} (\|x-p\|_{_R}<\varepsilon \\ \mbox{ \& } 
  \|a(x-p)-y\|_{_R}<\varepsilon)\}. 
  \end{align*}

\begin{Remark}
We can also give a valuative definition of $\mathcal C_p(X)$.
 \begin{align*}
  \mathcal C_p(X)=\{y\in R^n\ |\ \forall \lambda \in \Gamma \exists x\in X\exists a\in R_{>0} (\hat{v}(x-p)>\lambda \\ \mbox{ \& } 
  \hat{v}(a(x-p)-y)> \lambda\}. 
  \end{align*}
\end{Remark}

\begin{subsection}{Description of $\mathcal C_p(X)$ via curves}\label{curvedefinition}
We aim to present a description of the elements of $\mathcal C_p (X)$ when $X$ is an $\mathcal L_V$-definable subset of $R^n$ using $\mathcal L_V$-definable curves. 

We first discuss the case of the classical tangent cones as motivation. Suppose that $X$ is a semi-algebraic subset of $\mathbb R^n$. The tangent cone of $X$ at $p\in \mathbb R^n$ is defined by simply replacing $R$ with $\mathbb R$ in Definition~\ref{tcdef}. Let us denote this tangent cone by $C_p(X)$. Notice that the valuative definition of tangent cones does not apply to $C_p(X)$. Suppose that $y\in C_0(X)$. Using $\|\cdot\|$ to denote the usual norm on $\mathbb R^n$, $y/\|y\|$ belongs to the topological closure of the semi-algebraic set $U(X):=\{x/\|x\|\ |\ x\in X\}$. By the curve selection lemma for semi-algebraic sets (see Section 2.5 in~\cite{rag}) there exists a semi-algebraic injective curve $\gamma :(0,1)\longrightarrow U(X)$ for which $\gamma(t)\longrightarrow y/\|y\|$ as $t\longrightarrow 0^+$ (moreover, $\gamma$ could be assumed to be differentiable). Take $\eta:(0,1)\longrightarrow X$ to be the definable curve for which $\gamma(t)=\eta(t)/\|\eta(t)\|$ for all $t\in (0,1)$. Note that $y\in C_0(im\, \eta)$. 

The above argument shows that elements of $C_0(X)$  belong to the tangent cone (at 0) of definable (i.e. semi-algebraic) curves in $X$. By definable curve in $X$ we mean a definable injective curve $\eta$ with image contained in $X$. We will frequently refer to the image of such $\eta$ as the curve itself, so the function and its image set will be interchangeable. Let us summarise and enhance the previous facts. 

\begin{Proposition}\label{curvedef1}
 Let $X$ be a semi-algebraic subset of $\mathbb R^n$ and $p\in \mathbb R^n$. Then for $y\in \mathcal \mathcal C_p(X)$, there exists a semi-algebraic differentiable curve $\eta:(0,\varepsilon)\longrightarrow X$ 
  such that $\lim_{t\rightarrow 0^+}\eta(t)=p$ and  
  \[ \lim_{t\rightarrow 0^+}\eta'(t)=\lim_{t\rightarrow 0^+}\tfrac{\eta (t)-p}{t}=y.
  \]
\end{Proposition}
\begin{proof}
 We only need to reparametrise the curve $\eta$ described earlier, if needed.
\end{proof}

\begin{Remark}\label{remark}
This new description of $C_p(X)$ provides the freedom of choosing suitable sequences in the first definition of the tangent cone: for an element $y\in C_p(X)$, where $X$ is a semi-algebraic subset of $\mathbb R^n$ and $p\in \mathbb R^n$, and a sequence $(r_\mu)\subseteq \mathbb R_{>0}$ converging to 0, we can choose a sequence $(x_\mu)\subseteq X$ converging to $p$ in such a way that, as $\mu\longrightarrow \infty$, 
\[
r_\mu (x_\mu-p)\longrightarrow y.
\]
\end{Remark}

Now we come back to the non-archimedean setting. Our purpose for the rest of the section is to establish the analogous result to Proposition~\ref{curvedef1} for $\mathcal{L}_V$-definable subsets of $R^n$. The way to accomplish this is to prove an adequate curve selection lemma for such subsets. The o-minimality of $R$ as an $\mathcal L_{or}$-structure entails easily such result for $\mathcal L_{or}$-definable sets. In order to obtain the full statement, we exploit the results of L. van den Dries and A. Lewenberg in~\cite{driesII} and~\cite{driesI} on T-convexity. By (2.8) in~\cite{driesI}, $V$ is an $RCF$-convex subring of $R$ (here $RCF$ is the theory of real closed field in the language $\mathcal L_{or}$), this is in turn a consequence of the polynomially boundedness of $RCF$ (see~\cite{driesI} for more details). 

The next lemma is an easy consequence of Corollary 2.8 in~\cite{driesII}.
\begin{Lemma}\label{limitexistence}
Let  $f:R\longrightarrow R^n$ be a bounded $\mathcal{L}_{V}$-definable function and $a\in R$. Then $\lim\limits_{x\rightarrow a^{\scriptscriptstyle +}}f(x)$ exists in $R^n$.
\end{Lemma}
\begin{proof}
For a function as in the hypotheses, there exists a finite partition $(C_i)_{i\leq m}$ of $R$ such that:
\begin{itemize}
 \setlength\itemsep{-0.5ex}
 \item every $C_i$ is an $\mathcal{L}_V$-definable convex subset;
 \item for any $i\leq m$, there exists an $\mathcal L_{or}$-definable function $f_i:R\longrightarrow R^n$ for which $f_i|_{C_i}=f|_{C_i}$.
\end{itemize}
Using the o-minimality of $R$ in the language $\mathcal L_{or}$ we can assume that $f_i|_{C_i}$ is continuous (and strictly monotonic, as required in the proof of the next proposition). Thus, the limit of $f$ when $x\rightarrow a^{\scriptscriptstyle +}$ is the limit of $f_i|_{C_i}$ when $x\rightarrow a^{\scriptscriptstyle +}$, which exists in $(R\cup {\scriptstyle \{\pm \infty\}})^n$. Since $f$ is bounded, we can ensure that  $\lim\limits_{x\rightarrow a^{\scriptscriptstyle +}}f(x)\in R^n$.
\end{proof}

\begin{Proposition}\textup{(}\textit{Curve selection lemma for $(R,V)$}\textup{).}\label{plainselection}
 Assume $X$ is an $\mathcal{L}_V$-definable subset of $R^n$ and $x$ is an element of the topological closure of $X$. Then there exists $\varepsilon \in R_{>0}$ and an $\mathcal{L}_V$-definable curve $\gamma :(0,\varepsilon)\longrightarrow X$ such that $\lim_{t\rightarrow 0^+}\gamma(t)=x.$
\end{Proposition} 
\begin{proof}
 Let $c$ be a non-zero element of $R$ of positive valuation. By remark 2.7 in \cite{driesII}, the theory $RCF_{convex,c}$ has definable skolem functions. Using the proof of the previous lemma the result follows as usual: consider the set 
 \[
 A:=\{(t,y)\in R_{>0}\times X\ |\ \|x-y\|_{_R}<t\}.
 \] 
 Then $A$ is a definable subset of $R^{1+n}$. Since $x$ is an element of the topological closure of $X$, for every $t\in R_{>0}$ there exists $y\in X$ such that $(t,y)\in A$. Hence there exists a definable function $f:R_{>0}\longrightarrow X$ such that for each $t\in R_{>0}$, $(t,f(t))\in A$. From the proof of the previous lemma we deduce that there is $\varepsilon \in R_{>0}$ such that $f$ is continuos and injective on $(0,\varepsilon)$. Set $\gamma$ to be $f|_{(0,\varepsilon)}$. Notice that in this proof we may have added $c$ as a parameter in the definability of $\gamma$, but this does not matter much in applications.
\end{proof}

\noindent We are now prepared to obtain a version of the desired description of elements of $\mathcal C_p (X)$.

\begin{Corollary}\label{simplify}
 Let $X\subseteq R^n$ be $\mathcal L_V$-definable and fix $p\in  R^n$. If $y\in \mathcal C_p (X)$, then there exists an  $\mathcal{L}_V$-definable curve $Y\subseteq X$ such that $y\in \mathcal C_p (Y)$. 
\end{Corollary}
\begin{proof}
We could simply follow the proof we gave in the case of $X\subseteq \mathbb R^n$ being a semi-algebraic set. For the sake of showing an interesting possible treatment of the tangent cone, we present a different argument. As usual, we assume $p=0$. The set 
\[
D(X):=\{(x,r)\in R^n\times R_{>0}\ |\  rx\in X\}
\] 
is a definable set and it reproduces $\mathcal C_0(X)$ in the following sense,

\[
\overline{D(X)}\cap (R^n\times \{0\})=\mathcal C_0(X)\times \{0\},
\]
 where $\overline{D(X)}$ denotes the topological closure of $D(X)$.
If $y\in \mathcal C_0(X)$, then $(y,0)\in \overline{D(X)}$ and by Proposition~\ref{plainselection} there is a definable curve $\gamma :(0,1)\longrightarrow D(X)$ for which $\lim_{t\rightarrow 0^+}\gamma (t)=(y,0)$. Let $\gamma_1:(0,1)\longrightarrow R^n$ and $\gamma_2:(0,1)\longrightarrow R_{>0}$ be the definable curves that satisfy $\gamma (t)=(\gamma_1(t),\gamma_2(t))$ for all $t\in (0,1)$. Then the product $\gamma_1\cdot \gamma_2$ is an $\mathcal L_V$-definable curve in $X$ and $y\in \mathcal C_0(\gamma_1\cdot \gamma_2)$. 
\end{proof}

 As immediate consequence of the result above, we present now the desired description of $\mathcal C_p (X)$.
\begin{Proposition}\label{curvedef2}
Let $X\subseteq R^n$ be $\mathcal L_{V}$-definable and fix $p\in R^n$, then for $y\in \mathcal \mathcal C_p(X)$, there exists an $\mathcal L_V$-definable differentiable curve $\eta:(0,\varepsilon)\longrightarrow X$ 
  such that $\lim_{t\rightarrow 0^+}\eta(t)=p$ and  
  \[ 
  \lim_{t\rightarrow 0^+}\eta'(t)=\lim_{t\rightarrow 0^+}\tfrac{\eta (t)-p}{t}=y.
  \]
\end{Proposition}
\end{subsection}
\end{section}

\begin{section}{Stratifications of tangent cones}
We now introduce the second topic of this paper, stratifications in the setting of $R$. After defining the maps so-called risometries, we prove a strong result about the relation between $\mathcal L_V$-definable sets of $R$ and their tangent cones (Proposition~\ref{stronger}), this is a prompt application of Proposition~\ref{curvedef2}. In the last subsection we prove that t-stratifications induce t-stratifications on tangent cones. 

\begin{subsection}{Risometries, translatability and t-stratifications}\label{t-strats}
We first introduce some model-theoretic concepts needed below. We define $\text{RV}^{(n)}$ as the quotient of $R^n$ by the equivalence relation given by $\hat{v}(x-y)>\hat v(x)$ or $x=y=0$, for $x,y\in R^n$. The natural map is denoted by $\hat{rv}:R^n\longrightarrow \text{RV}^{(n)}$. $\text{RV}^{\text{eq}}$ is the set of all imaginaries defined from $\text{RV}^{(n)}$ ($n>0$), see  Definition 2.16 in~\cite{immi} for more details. 

By \emph{valuative open ball} we mean a set of the form $\{x\in R^n\ |\ \hat{v}(x-a)>\lambda\}$ where $a$ and $\lambda$ are fixed elements of $R^n$ and $\Gamma$, respectively; such set is denoted by $B(a,>\lambda)$. If the inequality is relaxed to $\geq$, then the set is called a \emph{valuative closed ball} and is denoted by $B(a, \geq\gamma)$. We will frequently drop the adjective ``valuative'' when no confusion is possible. By saying ``let $B$ be a ball'' we mean $B$ is a valuative ball open or closed.   
\begin{Definition}
Let $B$ be a ball in $R^n$. A \emph{risometry} on $B$ is a bijection $\varphi:B\longrightarrow B$ for which $\hat{rv}(\varphi (x)-\varphi (y))=\hat{rv}(x-y)$ holds for any $x,y\in B$. 
\end{Definition}

A risometry $\varphi$ is in particular an isometry, i.e. $\hat v(\varphi (x)-\varphi(y))=\hat v(x-y)$ holds for any $x,y$ in the domain of $\varphi$.

Let $B\subseteq R^n$ be a ball and $\chi:B\longrightarrow \text{RV}^{\text{eq}}$ a definable map. Definition 3.1 in~\cite{immi} provides us with a notion of \emph{translatability} of $\chi$ on $B$. The map $\chi$ being $d$-translatable on $B$ roughly means that it is almost invariant under translation by $\tilde V$, where $\tilde V$ is a $d$-dimensional subspace of $R^n$. Definable risometries are used to make precise this almost invariability. This version of translatability will not be used in this work, we use an equivalent version given by Lemma 3.7 in~\cite{immi}.

For $B$ a ball we put $B-B:=\{x-y\ |\ x,y\in B\}$. If $W$ is a subspace of dimension $d$ of $\overline{R}^n$, we say that the coordinate projection $\pi:R^n\longrightarrow R^d$ is an \emph{exhibition} of $W$ if, composed with the residue map, it induces an isomorphism between $W$ and $\overline{R}^d$. The projection at the level of the residue field induced by $\pi$ is denoted by $\overline{\pi}:\overline{R}^n\longrightarrow \overline{R}^d$. The map $\dir:R^n\setminus \{0\}\longrightarrow \mathbb G(\overline{R}^n)$ takes any $x$ to the subspace $\res(R\cdot x)$ of $\overline{R}^n$. We usually treat $\dir(x)$ as a representative of $\res(R\cdot x)$.

\begin{Definition}\label{translaters}
Let $B$, $\chi$ and $W$ be as above and $\pi:R^n\longrightarrow R^d$ be an exhibition of $W$. We say that $\chi$ is $W$-\emph{translatable} on $B$ (with respect to $\pi$) if there exists a definable family of risometries $(\alpha_x:B\longrightarrow B)_{x\in \pi(B-B)}$ with the following properties, for all $x,x'\in \pi(B-B)$ and $z\in B$,
\begin{enumerate}
\setlength\itemsep{0.2ex}
 \item $\chi\circ \alpha_x=\chi$;
 \item $\alpha_x\circ \alpha_{x'}=\alpha_{x+x'}$;
 \item $\pi(\alpha_x(z)-z)=x$;
 \item $\dir(\alpha_x(z)-z)\in W$.
\end{enumerate}
For $d\leq n$, we say that $\chi$ is $d$-\emph{translatable} on $B$ if there exists a $d$-dimensional subspace $W\subseteq \overline{R}^n$ and an exhibition $\pi$ of $W$ such that $\chi$ is $W$-translatable on $B$ (with respect to $\pi$). 
\end{Definition}

The family $(\alpha_x:B\longrightarrow B)_{x\in \pi(B-B)}$ above is called a \emph{translater} of $\chi$ on $B$ witnessing $W$-translatability.

It is worth mentioning that translatability does not depend on the choice of the projection $\pi$ exhibiting $W$, i.e. once we now $\chi$ is $W$-translatable on the ball $B$ (with respect to one particular exhibition of $W$), we can then find a family of risometries as in the definition for any given exhibition of $W$.  

For $A\subseteq R^n$, $\chi_A:R^n\longrightarrow \text{RV}^{\text{eq}}$ is the function given by $\chi_A(x)=0$ if $x\in A$ and $\chi(x)=1$ otherwise. For a tuple $(A_0,\dots, A_l, \rho_{l+1},\dots, \rho_m)$ of subsets and maps into $\text{RV}^{\text{eq}}$ we put them into a sigle map $\chi :=(\chi_{A_0},\dots ,\chi_{A_l}, \rho_{l+1},\dots, \rho_m)$. Notice that if all the subsets $A_i$ and maps $\rho_i$ are definable, so is $\chi$. Thus we can apply the previous definitions to definable subsets and tuples of definable subsets and maps using this trick. 

Now we can define t-stratifications. By a definable partition of a set we mean that all the sets conforming the partition are definable.  

\begin{Definition}
 Let $B_0\subseteq R^n$ be a ball or $R^n$. A (finite) definable partition $(S_i)_i=(S_0,\dots, S_n)$ of $B_0$ is said to be a \emph{t-stratification} of $B_0$ if the following holds:
 \begin{enumerate}
  \item For all $d\leq n$, $\dim(\bigcup_{j\leq d}S_j)\leq d$;
  \item For each $d\leq n$ and each ball $B\subseteq S_d\cup \dots \cup S_n$,  $(S_i)_i$ is $d$-translatable on $B$. 
  
 If $\chi:B_0\longrightarrow  \text{RV}^{\text{eq}}$ is a definable map, we say that the t-stratification $(S_i)_i$ of $B_0$ \emph{reflects} $\chi$ if the following strengthening of 2. holds.
 \item[2'.] For each $d\leq n$ and each ball $B\subseteq S_d\cup \dots \cup S_n$, $((S_i)_i,\chi)$ is $d$-translatable on $B$. 
 \end{enumerate}
\end{Definition}

\noindent By saying that $(S_i)_i$ is a t-stratification of the subset $A\subseteq B_0$ we mean that $(S_i)_i$ reflects the map $\chi_A$ as defined before (and this applies to tuples of subsets and maps as well). 

It is worth mentioning that t-stratifications can be defined in an even more general setting, that is, the one of valued fields of equi-characteristic 0 (notice that $R$ and $\overline R$ being real closed fields ensure such condition in our setting). Accordingly, from the point of view of their existence, it was proved by I. Halupczok in~\cite{immi} that given a formula $\phi(x,y)$ that defines a map into $\text{RV}^{\text{eq}}$, there exist a finite family of formulas $(\psi_i(z))_i$ such that in any suitable Henselian valued field $K$ of equi-characteristic 0, $(\psi_i(K^n))_i$ is a t-stratification of $K^n$ reflecting the map defined by $\phi(x,y)$ on $K^n$. This basically means that given any finite family of definable maps and sets, one can always obtain a t-stratification reflecting such family and defined uniformally on all suitable fields $K$. The suitable valued fields in focus are the models of a theory $\mathcal T$ satisfying all the requirements in Section 2.5 of~\cite{immi}. In particular, $\mathcal T$ contains the theory of all Henselian valued fields of equi-characteristic 0 in the language $\mathcal L_{\text{Hen}}$ described in Definition 2.16 in~\cite{immi}. An example of such $\mathcal T$ is given in our setting. This is deduced from the following theorem and Example 5.1 in~\cite{immi}, since the latter states that in our setting Hypothesis 2.21 in~\cite{immi} is satisfied.

 \begin{Theorem}[I. Halupczok~\cite{immi}]\label{generalexist}
Let $\mathcal L$ be a language containing $(\mathcal L_V)_{\text{Hen}}$ and $\mathcal T$ be an $\mathcal L$-theory containing the theory of a real closed field in the language $(\mathcal L_V)_{\text{Hen}}$ and satisfying Hypothesis 2.21 in~\cite{immi}. For any $\mathcal L$-formula $\phi(x)$ such that $\phi(R^n)$ defines a map into $\text{RV}^{\text{eq}}$ for any $R\vDash \mathcal T$, there are $\mathcal L$-formulas $(\psi_i(x))_{i\leq n}$ such that for any $R\vDash \mathcal T$, $(\psi_i(R^n))_i$ is a t-stratification reflecting $\phi(R^n)$. 
 \end{Theorem}
 
 For most of the paper we are only interested in $\mathcal L=(\mathcal L_V)_{\text{Hen}}$ and $\mathcal T=$ the theory of a real closed field in the language $(\mathcal L_V)_{\text{Hen}}$, which correspond to our setting so far. Only in Section~\ref{comments} we discuss further possibilities of our results in more general $\mathcal L$ and $\mathcal T$. 

\end{subsection}
 
 \begin{subsection}{Risometries between tangent cones}\label{riso}
 During the whole of this subsection, $B$ will denote the ball $B(0,>0):=\{x\in R^n\ |\ \hat v(x)>0\}$. The following theorem reflects a strong relation between sets and their tangent cones. The overline notation $\overline{X}$ denotes the topological closure of the set $X$.

\begin{Proposition}\label{stronger} Let $X,Y\subseteq R^n$ be $\mathcal{L}_V$-definable and fix $p\in \overline{X}\cap \overline{Y}$. If there exists an $\mathcal{L}_V$-definable risometry $\varphi:B\longrightarrow B$ taking $X\cap B$ to $Y\cap B$ and fixing $p$, then there exists an $\mathcal{L}_V$-definable risometry $\psi$ on $B$ taking $\mathcal C_p(X)\cap B$ to $\mathcal  C_p (Y)\cap B$ that fixes $0$.
\end{Proposition}

\begin{proof}
First we present a general construction of a risometry $\psi:B\longrightarrow B$ from the risometry $\varphi$, then we deduce the result from this. We assume $p=0$ below.

Consider $x\in B$ and a definable curve $\gamma:(0,\varepsilon)\longrightarrow B$ for which $x$ is the limit of $\gamma(t)/t$ as $t\rightarrow 0^+$. Suppose that the composition $\varphi \circ \gamma$, which is definable and, hence, can be regarded as an injective curve, is such that $t^{-1}\varphi \circ \gamma(t)$ is bounded. By Lemma~\ref{limitexistence}, the limit of $t^{-1}\varphi \circ \gamma(t)$ as $t\rightarrow 0^+$ exists in $B$. We set 
\[
\psi(x):=\lim_{t\rightarrow 0^+}\frac{\varphi \circ \gamma (t)}{t}.
\]
\par 
We claim that $\psi:B\longrightarrow B$ is a definable risometry. First of all, that $\psi$ is well defined is an immediate consequence of $\varphi$ being, in particular, an isometry. Indeed, if $\eta:(0,\varepsilon)\longrightarrow B$ is another definable curve for which $\eta(t)/t$ converges to $x$, then the following equation holds for any $t\in (0,\varepsilon)$,
\[
\hat v\left(\frac{\varphi\circ\gamma(t)}{t}-\frac{ \varphi \circ \eta(t)}{t}\right)=\hat v\left(\frac{\gamma (t)}{t}-\frac{\eta(t)}{t}\right),
\]
from which it is clear that, $t^{-1}\varphi \circ \gamma (t)$ and $t^{-1}\varphi \circ \eta (t)$ have the same limit as $t\rightarrow 0^+$. 
To show that $\psi$ is bijective, let us consider the map $\psi'$ built from $\varphi ^{-1}$ in the analogous way as $\psi$ was built from $\varphi$. For $x\in B$, let us pick a definable curve $\gamma$ such that $\gamma(t)/t$ converges to $x$ as $t\rightarrow 0^+$. For what follows we assume that $\varphi \circ \gamma$ is already a definable curve, without loss of generality. Then 
\begin{align*}
\psi'\circ \psi(x)=\psi'\left(\lim\limits_{t\rightarrow 0^+}\frac{\varphi \circ \gamma (t)}{t}\right) 
=\lim\limits_{t\rightarrow 0^+}\frac{\varphi ^{-1}\circ \varphi \circ \gamma (t)}{t}\\
=\lim\limits_{t\rightarrow 0^+}\frac{\gamma (t)}{t}=x, 
\end{align*}
where we have used that the maps do not depend on the curve considered. Similarly, we can check that $\psi\circ \psi'$ is also the identity on $B$. Therefore, $\psi$ is bijective. 

Finally, we can now check that the risometric condition is satisfied by $\psi$. For this we first highlight the following easy fact: if $0\neq x$ is the limit of $f(t)$ when $t\rightarrow 0^+$, with $f$ any injective function, then for $t>0$ sufficiently small, $\hat{rv}(f(t))=\hat{rv}(x)$. Let $x, y\in X$ and let $\gamma, \eta:(0,\varepsilon)\longrightarrow B$ be definable curves for which $\lim\limits_{t\rightarrow 0^+}\frac{\gamma (t)}{t}=x$ and $\lim\limits_{t\rightarrow 0^+}\frac{\eta (t)}{t}=y$. Then, for $t$ sufficiently small,
\begin{align*}
 \hat{rv}(\psi(x)-\psi(y))=\hat{rv}\left(\frac{\varphi \circ \gamma(t)}{t}-\frac{\varphi \circ \eta(t)}{t}\right)\\
 =\hat{rv}\left(\frac{\gamma(t)}{t}-\frac{\eta(t)}{t}\right)=\hat{rv}(x-y).
\end{align*}

Thus $\psi$ is a risometry. Notice, furthermore, that this construction allows us to make $\psi$ take $0$ to $0$, as $\varphi$ does. 

Now we deduce the result. This follows from the fact that $\psi$ is well defined: this means that if we consider a point $y$ in $\mathcal C_0 (X)\cap B$, for the definition of  $\psi(x)$ we can certainly use a curve in $X$ as the choice of this curve does not matter. Also notice that in this way $\psi(y)\in \mathcal C_0 (\varphi(X\cap B))=\mathcal C_0 (Y)\cap B$. So the same construction goes through with the particularity that points in $\mathcal C_0 (X)\cap B$ are sent to elements in $\mathcal C_0 (Y)\cap B$.
\end{proof}
 \end{subsection} 

\begin{subsection}{T-stratifications induced on tangent cones}
Assume that $(S_i)_i$ is a t-\-stra\-ti\-fi\-cation in $R^n$. Let $S_{\leq i}$ stand for the set $\bigcup _{j\leq i}S_j$.

\begin{Definition} \label{conestrata}
For fixed $p\in R^n$, the partition $(\mathcal C_{p,i})_i$ of $R^n$  is defined as follows.
 \[
 \mathcal C_{p,0}=\mathcal C_p(S_0)
\]
and for $0\leq i<n$, 
\[
\mathcal C_{p,i+1}=\mathcal C_p(S_{\leq i+1})\setminus \mathcal C_p(S_{\leq i}).
\] 
\end{Definition}

Since by definition  $(S_i)_i$ is an $\mathcal L_V$-definable partition, each stratum $\mathcal C_{p,i}$ is an $\mathcal L_V$-definable set. If $X$ is a subset of $R^n$ and $(S_i)_i$ is a t-stratification of $X$, we informally say that $(S_i)_i$ \emph{induces} a  t-stratification on $\mathcal C_p(X)$ if  $(\mathcal C_{p,i})_i$ turns out to be a t-stratification of $\mathcal C_p(X)$.

 We now prove that t-stratifications induce t-stratifications on tangent cones. First we establish an auxiliary lemma, which is related to Corollary 3.22 in~\cite{immi} and will be used to guarantee the existence of definable translaters later. The notation $\ulcorner B\urcorner$ stands for the code (i.e. canonical parameter) of the ball $B$ in $\text{RV}^{\text{eq}}$ (see Section 2.4 in~\cite{immi}).

\begin{Lemma}\label{corollary3.22}
 Let $(S_i)_i$ be a t-stratification reflecting the $\mathcal L_V$-definable map $\rho:R^n\longrightarrow \text{RV}^{\text{eq}}$, $d\leq n$ and $B\subseteq R^n$ be a ball maximal such that $S_{\leq d-1}\cap B=\emptyset$. If $W:=tsp_B((S_i)_i,\rho)$ is exhibited by the projection $\pi:R^n\longrightarrow R^d$, then there is a $\ulcorner B \urcorner$-definable translater $(\alpha_x)_{x\in \pi(B-B)}$ witnessing the $W$-translatability of $((S_i)_i,\rho)$ on $B$.
\end{Lemma}
\begin{proof}
 Take $Q:=\pi(R^n)$ and $\chi:Q\times R^{n-d}\longrightarrow \text{RV}^{\text{eq}}$ be given by $\chi(q,x)=\rho(q\hat{\ }x)$. Also set $S_{i,q}=S_i\cap \pi^{-1}(q)$ and $\chi_q=\chi|_{\pi^{-1}(q)}$ for any $q\in Q$. Since $(S_i)_i$ is already a t-stratification reflecting $\chi$, all the hypotheses of Proposition 3.19 in~\cite{immi} are satisfied. By (3) of the mentioned proposition, there is a compatible $\ulcorner B\urcorner$-definable family of risometries $(\alpha_{q,q'}:((S_{i,q})_i,\chi_q)\longrightarrow ((S_{i,q})_i,\chi_{q'}))_{q,q'\in \pi(B)}$, respecting $((S_i)_i,\chi)$. Furthemore, notice that by the choice of $W$ and $\pi$, $(S_i)_i$ is $W$$-\pi$-pointwise translatable on $\pi(B)\times R^{n-d}$ and $S_d\cap (\pi(B)\times R^{n-d})\neq \emptyset$, so (3') of Proposition 3.19 in~\cite{immi} holds too. This is, for all $q,q'\in \pi(B)$ and $z\in \{q\}\times R^{n-d}$, $\dir(\alpha_{q,q'}(z)-z)\in W$.
 
 For $x\in \pi(B-B)$ and $z\in B$, let $q:=\pi(z)$ and then set $\alpha_{x}(z):=\alpha_{q,q+x}(z)$. This defines the maps $(\alpha_{x}:B\longrightarrow B)_{x\in \pi(B-B)}$. Below we show that they satisfy the conditions (i)-(iv) of being a translater of $((S_i)_i,\rho)$ on $B$ (see Definition~\ref{translaters}). Aftewards we show that they are risometries. Let $x,x'\in \pi(B-B)$, $z\in B$ and set $q=\pi(z)$.
 \begin{enumerate}[(i)]
 \setlength\itemsep{0.2ex}
  \item $((S_i)_i,\rho)\circ \alpha_x=((S_i)_i,\rho)$.\\
  This holds because each risometry $\alpha_{r,r'}$ respects $((S_i)_i,\rho)$.
  \item $\alpha_{x'}\circ \alpha_{x}=\alpha_{x+x'}.$\\
  Following the definitions of the pertinent maps the equation is equivalent to the following one, 
  \begin{align*}
   \alpha_{q+x,q+x+x'}\circ \alpha_{q,q+x}(z)=\alpha_{q,q+x+x'}(z),
  \end{align*}
  and this follows at once from the compatibility of the maps $(\alpha_{r,r'})_{r,r'\in \pi(B)}$.
  \item $\pi(\alpha_{x}(z)-z)=x$.\\
  Since $\alpha_{q,q+x}$ takes the set $\{q\}\times R^{n-d}$ to $\{q+x\}\times R^{n-d}$, $\pi(\alpha_{q,q+x}(z))=q+x$. So 
  \[
  \pi(\alpha_{x}(z)-z)=\pi(\alpha_{q,q+x}(z)-z)=x.
  \]
  \item $\dir(\alpha_{x}(z)-z)\in W$, if $x\neq 0$.\\
  Since (3') of Proposition 3.19 in~\cite{immi} holds, we have that 
  \[ 
\dir(\alpha_x(z)-z)=\dir(\alpha_{q,q+x}(z)-z)\in W.
\]  
 \end{enumerate}
To check that $\alpha_x$ is a risometry, let $z,z'\in B$ and set $q=\pi(z)$ and $q'=\pi(z')$ with $q\neq q'$. We want to verify that $\hat{rv}(\alpha_{x}(z)-\alpha_x(z'))=\hat{rv}(z-z')$.

Let $w=\alpha_{q,q'}(z)$, so $\pi(w)=q'$. Since $\alpha_{q',q'+x}$ is a risometry, 
\begin{equation}
 \hat{rv}(\alpha_{q',q'+x}(w)-\alpha_{q',q'+x}(z'))=\hat{rv}(w-z').
\end{equation}
We also claim that 
\begin{equation}
 \hat{rv}(\alpha_{q,q+x}(z)-\alpha_{q',q'+x}(w))=\hat{rv}(z-w).
\end{equation}
The required equation follows using (1) and (2) and Lemma 2.10 in~\cite{immi}. To apply such lemma we need to check that $\hat v(z-z')=\min \{\hat v(z-w),\hat v(w-z')\}$ holds. Suppose it does not, then $\hat v(z-w)=\hat v(w-z')$ and, hence, $\dir(z-w)=\dir(w-z')$. Since $\pi$ is an exhibition of $W$ and $\dir(z-w)=\dir(z-\alpha_{q,q'}(z))\in W$ we would have that $q-q'=0$, a contradiction. 

Thus it only remains to prove (2). By the choice of $\pi$, it is enough to show that both $\pi(\alpha_{q,q+x}(z)-\alpha_{q',q'+x}(w))=\pi(z-w)$ and $\dir(\alpha_{q,q+x}(z)-\alpha_{q',q'+x}(w))=\dir(z-w)$ hold (again using Lemma 2.10 in~\cite{immi}).
Indeed, 
\[
\pi(\alpha_{q,q+x}(z)-\alpha_{q',q'+x}(w))=(q+x)-(q'+x)=q-q'=\pi(z-w);
\]
and for the second equation recall that $\pi$ is an exhibition of $W$, so we have
\begin{align*}
\overline{\pi} (\dir(\alpha_{q,q+x}(z)-\alpha_{q',q'+x}(w))=\dir \pi(\alpha_{q,q+x}(z)-\alpha_{q',q'+x}(w))\\
=\dir \pi(z-w)=\overline{\pi}(\dir (z-w)),
\end{align*}
where $\overline{\pi}$ is the projection induced by $\pi$ on the residue field, hence both equations hold.
\end{proof}

Now we prove Theorem~\ref{intro1}.

\begin{Theorem}\label{main}
Let $X$ be an $\mathcal{L}_V$-definable subset of $R^n$, $p\in R^n$ and suppose that $(S_i)_i$ is a t-stratification of $X$. Then $(\mathcal  C_{p,i})_i$ is a t-stratification of $\mathcal C_p (X)$.
\end{Theorem}
\begin{proof}
 For any subset $A$ of $R^n$, $\dim(A)$ is the maximal $m\geq 0$ such that there is a coordinate projection $\tau:{R}^n\longrightarrow R^m$ for which $\tau (A)$ contains a ball. Also notice that for such $\tau$ and $A$, $\tau ({C}_p (A))={C}_p (\tau(A))$. Take $1\leq d\leq n-1$, then we know that $\dim(S_{\leq d})\leq d$. From these facts it is easy to see that $\dim({C}_p (S_{\leq d}))\leq d$. By the properties of dimension we deduce that 
 \[
 \dim({C} _{p,d})=\dim ({C}_p (S_{\leq d})\setminus {C}_p (S_{\leq d-1}))\leq d,
 \] 
 as required.
 
 For simplicity now we assume that $p=0$. Fix $1\leq d\leq n$ and let $B:=B(u, >\mu)$ be maximal such that $B\cap {C} _{0,\leq d-1}=\emptyset$, with $u\in  C_0(S_d)$. We claim that $({C} _{0,i})_i$ is $d$-translatable on $B$. To prove this, let us fix a definable curve $\gamma:(0,1)\longrightarrow {}^*S_d$ for which $\gamma(t)/t$ converges to $u$ when $t\rightarrow 0^+$. For every $t\in (0,1)$, set $u_t:=\gamma(t)$ and $B_t:=B(u_t,>\mu +v(t))$. 
 
 We can investigate whether $(S_i)_i$ is $d$-translatable on the balls $B_t$. Assume that $(S_i)_i$ is not $d$-translatable on $B_t$ for almost all $t$. That $d$-translatability of the t-stratification $(S_i)_i$ is a definable condition is a consequence of Proposition 3.19(1) in~\cite{immi}, so by weakly o-minimality of $R$ we may assume that $(S_i)_i$ is not $d$-translatable for all $0<t<\varepsilon$, for some small $\varepsilon>0$. For such $t$ we deduce that $B_t\cap S_{\leq d-1}\neq \emptyset$. The definability of $B_t\cap S_{\leq d-1}$ (taking $t$ as a parameter), the existence of definable Skolem functions for $RCF_{convex,c}$ and the proof of Lemma~\ref{limitexistence}, imply that there is a definable curve $\eta$ such that $\eta(t)\in B_t\cap S_{ d-1}$ for all $t$ in the domain of $\eta$ and $\lim_{t\rightarrow 0^+}\eta(t)=0$. From these conditions it follows that the quotient $\eta(t)/t$ is bounded, so by Lemma~\ref{limitexistence} the limit of $\eta(t)/t$, when $t\rightarrow 0^+$, exists; let $y$ be such limit. It is clear that $y\in {C} _{0,\leq d-1}$ and since 
 \[
 \hat v(y-u)=\hat v\left(\frac{y_t}{t}-\frac{x_t}{t}\right)>\mu,
 \] 
we have that $y \in B\cap {C} _{0,\leq d-1}$, a contradiction. 
 
 Hence, it must be true that $(S_i)_i$ is $d$-translatable on $B_t$ for almost all $t$. Consider the collection $\{tsp_{B_t}(S_i)_i\}_t$ of subspaces of $\mathbb R^d$. There must exist a subspace $W_0\subseteq \mathbb R^d$ for which $W_0=tsp_{B_t}(S_i)_i$ for all $t$ sufficiently small. Consider the definable set $E\subseteq \Gamma\times \mathbb G_d(\mathbb R^n)$ given by $(\lambda,W)\in E$ if and only if there is $t\in R_{>0}$ such that $v(t)=\lambda$ and $tsp_{B_t}(S_i)_i=W$. The power boundedness of $R$ as an o-minimal $\mathcal L_{or}$-structure implies that $\Gamma$ is o-minimal (Proposition 4.3 in~\cite{driesII}). The orthogonality between $\Gamma$ and $\mathbb R$ and the o-minimality of $\Gamma$ imply that there are $\lambda_0\in \Gamma$ and $W_0\in \mathbb G_d(\mathbb R^n)$ such that for all $\lambda\geq \lambda _0$, $(\lambda,W_0)\in E$ and whenever $(\lambda, W)\in E$, $W=W_0$. Take $t_0\in R_{>0}$ such that $v(t_0)=\lambda$; so $W_0=tsp_{B_{t_0}}(S_i)_i$ and for all $0<t<t_0$, $tsp_{B_t}(S_i)_i=W_0$. 
 
%
%
 We will show that $({C} _{p,i})_i$ is $W_0$-translatable on $B$. Fix an exhibition $\pi$ of $W_0$. By Lemma~\ref{corollary3.22}, for every $t\in (0,t_0)$, there is a $\ulcorner B_t\urcorner $-definable translater $(\alpha_{t,x})_{ x\in \pi  (B_t-B_t)}$ witnessing the $W_0$-translatability of $(S_i)_i$ on $B_t$. Since being a translater of a t-stratification is expressable in a first order way, a compactness argument ensures that there are finitely many formulas defining all such translaters (these formulas take at least $t$ as parameter). 
 
 The next step is to define a translater $(\alpha _x)_{x\in \pi(B-B)}$ of $( C_{p,i})_i$ on $B$ witnessing $W_0$-translatability. For $z\in B$ and $t\in (0,t_0)$ set $z_t:=u_t+t(z-u)$. Then for any $t\in (0,t_0)$, $\hat v(u_t-z_t)=\hat v(u-z)+ v(t)>\mu + v(t)$, so $z_t\in B_t$; furthermore, $\hat v(t^{-1}z_t-z)=\hat v(t^{-1}u_t-u)$, so $t^{-1}z_t$  converges to $z$. Let $x\in \pi (B-B)$, then for any $t\in (0,t_0)$, $\alpha_{t,tx}(z_t)\in B_t$. Notice that $t\mapsto t^{-1}\alpha_{t,tx}(z_t)$ is a definable mapping thanks to the fact that all the translaters $(\alpha_{t,x})_{x\in \pi(B_t-B_t)}$ are defined by finitely many formulas. Thus Lemma~\ref{limitexistence} allows us to set 
  \[
  \alpha_x(z):=\lim _{t\rightarrow 0}\, t^{-1}\alpha_{t,tx}(z_t)
  \]
  as a function into $R^n$. This function is well defined, i.e. the limit does not depend intrinsically on the curve $t\mapsto z_t$. To show this suppose that $\eta:(0,t_0)\longrightarrow \bigcup_{t} B_t$ is a curve for which $\eta(t)\in B_t$ for all $t\in (0,t_0)$. Since each risometry $\alpha_{r,r'}$ is in particular an isometry, the following equation holds for every $t\in (0,t_0)$,
  \[
  \hat{v}\left(\frac{\alpha_{t,tx}(z_t)}{t}-\frac{\alpha_{t,tx}\circ \eta(t)}{t}\right)=\hat v\left(\frac{z_t}{t}-\frac{\eta(t)}{t}\right);
  \]
  it is, then, clear that $t^{-1}\alpha_{t,tx}(z_t)$ and $t^{-1}\alpha_{t,tx}(\eta(t))$ have the same limit when $t\rightarrow 0^+$.
  
  A quick calculation shows that $\alpha_x(z)\in B$ and for $z,z'\in B$ the equations  
  \begin{align*}\hat{rv}(\alpha_x(z)-\alpha_x(z')) =\hat{rv}(t^{-1}\alpha_{t,tx}(z_t)-t^{-1}\alpha_{t,tx}(z_t'))\\ 
  =\hat{rv}(t^{-1}z_t-t^{-1}z_t') =\hat{rv}(z-z'),\end{align*} 
  exhibit that $\alpha_x$ satisfies the risometric condition.
  
  Moreover, consider repeating the construction of $\alpha_x$ with the translaters \\$(\alpha^{-1}_{t,x})_{x\in \pi(B_n-B_n)}$. This gives us a map $\beta_x:B\longrightarrow B$. That the construction provides a well defined map implies that, if $z\in B$, 
  \begin{align*}
   \beta_x\circ \alpha_x(z)=\beta_x(\lim\limits_{t\rightarrow 0}t^{-1}\alpha_{t,tx}(z_t))
   =\lim\limits_{t\rightarrow 0}t^{-1}\alpha^{-1}_{t,tx}(\alpha_{t,tx}(z_t))=z.
  \end{align*}
Similarly one shows that $\alpha_x\circ \beta_x$ is the identity on $B$ as well. Thus the maps $\alpha_x:B\longrightarrow B$ are risometries. 
  
  We now show that the family of risometries $(\alpha _x)_{x\in \pi(B-B)}$ satisfy the required properties of a translater of $( C_i)_i$ on $B$. Let $z\in B$, $x,x'\in \pi (B-B)$ and pick a definable curve $t\mapsto z_t\in B_t$ such that $t^{-1}z_t$ converges to $z$ when $t\rightarrow 0^+$.
  \begin{enumerate}[(i)]
  \setlength\itemsep{0.2ex}
   \item $z\in  C_d$ if and only if $\alpha_x(z)\in  C_d$.\\
   Suppose that $z\in  C_d$. Since the function $\alpha_x$ is well defined, the curve  $t\mapsto z_t$ can be supposed to being taken inside $S_d$. So, for $x\in \pi(B-B)$, $\alpha_{t,tx}(z_t)\in S_d$, as $(\alpha_{t,x})_{x\in \pi(B_n-B_n)}$ is a translater of $(S_i)_i$. Thus $\alpha_x(z)$, being the limit of $t^{-1}\alpha_{t,tx}(z_t)$, is in $ C_d$. The converse statement follows similarly.
   \item $\alpha_x\circ \alpha _{x'}=\alpha_{x+x'}$.\\
   Using that the maps $\alpha_x$ are well defined,
   \begin{align*}
   \alpha_{x}\circ \alpha_{x'}(z)=\lim\limits_{t\rightarrow 0}t^{-1}\alpha_{t,tx}(\alpha_{t,tx'}(z_t)) =\lim\limits_{t\rightarrow 0}t^{-1}\alpha_{t,tx}\circ \alpha_{t,tx'}(z_t)\\ =\lim\limits_{t\rightarrow 0}t^{-1}\alpha_{t,t(x+x')}(z_n)=\alpha_{x+x'}(z).
   \end{align*}
   \item $\pi (\alpha_x(z)-z)=x$.\\
   We can see that, 
   \begin{align*}
    \pi(\alpha_x(z)-z)=\pi\left(\lim_{t\rightarrow 0}t^{-1}\alpha_{t,tx}(z_t)-t^{-1}z_t\right)\\ 
    =\lim_{t\rightarrow 0}t^{-1}\pi(\alpha_{t,tx}(z_t)-z_t) 
    =\lim_{t\rightarrow 0}t^{-1}(tx)=x.
   \end{align*}
   \item $dir(\alpha_x(z)-z)\in W_0$.\\
   Notice that for $t$ sufficiently small it holds that,
   \begin{align*}
   \dir(\alpha_x(z)-z)=\dir(t^{-1}\alpha_{t,tx}(z_t)-t^{-1}z_t)\\
  =\dir(\alpha_{t,tx}(z_t)-z_t)\in W_0. 
   \end{align*}
  \end{enumerate}

  This finishes the proof of $(\alpha_{x})_{x\in \pi(B-B)}$ being a translater of $({C}_i)_i$ on $B$ witnessing $W_0$-translatability. Thus, $({C}_i)_i$ is $d$-translatable on $B$.
 \end{proof}  
 \end{subsection}
 \end{section}
 
\begin{section}{Classical tangent cones and archimedean t-stratifications}\label{classical}
We now turn our attention to classical tangent cones and the archimedean couterpart of t-\-stra\-tifi\-cations. We remind the reader that by classical tangent cones we mean tangent cones in Euclidean space. The real field $\mathbb R$ is considered as a structure in the language $\mathcal L_{or}$. The $\mathcal L_{or}$-definable subsets of $\mathbb R^n$ are exactly the semi-algebraic sets.

 In order to transfer some of our previous results to $\mathbb R^n$, we let $R$  be a non-standard model of $\mathbb R$. Let $^*\mathbb R$ be a non-principal ultrapower of $\mathbb R$ (ensuring $\aleph_1-$saturation).  We keep roughly the same notation as for $R$ and, accordingly, we focus on the valuative structure of $^*\mathbb R$ in the language $\mathcal L_V$. In this case we can concretely make $V=\{x\in {}^*\mathbb R\ |\ \exists N\in \mathbb N(-N\leq x\leq N)\}$.
 
For a subset $X$ of $\mathbb R^n$, $^*X$ denotes its non-standard version in $^*\mathbb R^n$. For $x\in \mathbb R^n$ we identify it with its image under the canonical embedding of $\mathbb R^n$ into $^*\mathbb R^n$; so $^*x=x$ for such a point. Finally, recall that we denote by $\|\cdot\|$ the usual norm on $\mathbb R^n$ and we let $\|\cdot\|_{_{}^*\mathbb R}$ be the norm it induces on $^*\mathbb R^n$ (which is the same $^*\mathbb R$ carries naturally as real closed field).

 We first explore some facts connected to that in Proposition~\ref{stronger}. Accordingly with previous notation, the tangent cone of $Y\subseteq {}^*\mathbb R^n$ at $p\in {}^*\mathbb R^n$ is going to be $\mathcal C_p(Y)$, meanwhile the tangent cone of $X\subseteq \mathbb R^n$ at $p\in \mathbb R^n$ is $C_p(X)$. It is worth noting that when $X\subseteq \mathbb R^n$ is semi-algebraic and $p\in \mathbb R^n$, $^*C_p(X)=\mathcal C_p({}^*X)$. The following proposition bears a similar fact.

\begin{Proposition}\label{description}
 Suppose that $X\subseteq \mathbb R^n$ is semi-algebraic and $p\in \mathbb R^n$. Then, for $y\in \mathbb R^n$, $y\in \mathcal C_{p}({}^*X)$ iff  \[\exists x\in {}^*X,a\in \mathbb {}^*\mathbb R_{>0}\left(\hat{v}(x-p)>0\wedge \hat{v}(a(x-p)-y)>0\right).\]
\end{Proposition}
\begin{proof}
 For simplicity, we will assume that $p=0$. Let us suppose that $y\in \mathbb R^n$ satisfies the formula on the right-hand side. For any positive integer $n$ the formula 
 \[
 \exists x\in {}^*X,a\in {}^*\mathbb R_{>0}\left(\|x\|_{_{}^*\mathbb R}<\tfrac{1}{n}\wedge \|ax-y\|_{_{}^*\mathbb R}<\tfrac{1}{n}\right)
 \]
 is true in $^*\mathbb R$ (here we use that the valuation topology on $^*\mathbb R^n$ coincides with the topology induced by $\|\cdot\|$). By the transfer principle, the formula 
 \[
 \exists x\in X,a\in \mathbb R_{>0}\left(\|x\|<\tfrac{1}{n}\wedge \|ax-y\|<\tfrac{1}{n}\right)
 \] 
 is true in $\mathbb R$ for any positive integer $n$. Therefore $y\in {}^*C_0(X)=\mathcal C_0(^*X)$. Notice that all the steps in this argument are reversible.
\end{proof} 

Using this result and Proposition~\ref{stronger}, we prove a relation between tangent cones. Recall the notation $\overline{X}$ stands for the topological closure of $X$. In the following proposition and its proof, $B$ denotes the ball $\{x\in {}^*\mathbb R^n\ |\ \hat{v}(x)>0\}$.

\begin{Proposition}\label{first}
Assume $X,Y\subseteq \mathbb R^n$ are semi-algebraic sets, $p\in \overline{X}\cap \overline{Y}$ and that there is a definable risometry $\varphi$ on $B$ taking $^*X\cap B$ to $^*Y\cap B$ and fixing $p$. Then $C_p(X)=C_p(Y)$. 
\end{Proposition}

\begin{proof}
 By Proposition~\ref{stronger}, there is a risometry between $\mathcal C_p({}^*X)\cap B$ and $\mathcal C_p({}^*Y)\cap B$. The proof is completed by the following claim.
 
\textsc{Claim.} 
Let $X,Y\subseteq \mathbb R^n$ be semi-algebraic sets and fix $p\in \overline{X}\cap \overline{Y}$. If there is a risometry on $B$ taking $\mathcal C_p({}^*X)\cap B$ to $\mathcal C_p({}^*Y)\cap B$ and fixing $p$, then $C_p(X)=C_p(Y)$.  

 Put $p=0$ and let $\psi$ be the risometry in the hypothesis of the claim. For $x\in \mathcal C_0({}^*X)\cap V^n$ with $v(x)=0$, take $r\in {}^*\mathbb R_{>0}$ such that $rx\in B$. Then $\psi(rx)\in \mathcal C_0({}^*Y)$ and we can pick $s\in {}^*\mathbb R$ such that $\hat{rv}(s\psi(rx))=\hat{rv}(x)$; this is possible because $\hat{rv}(\psi(rx))=\hat{rv}(rx)$, using that $\psi$ is a risometry fixing 0. This shows that $\hat{rv}(\mathcal C_0({}^*X)\cap V^n)=\hat{rv}(\mathcal C_0({}^*Y)\cap V^n)$; whence $C_0(X)=\res(\mathcal C_0({}^*X)\cap V^n)=\res(\mathcal C_0({}^*Y)\cap V^n)=C_0(Y)$. 
\end{proof}

Now we turn our attention to the archimedean stratifications induced by t-stratifications in $^*\mathbb R$. 

\begin{Definition}
 Let $X$ be a subset of $\mathbb R^n$ and suppose that $(S_i)_i$ is a partition of $\mathbb R^n$ into semi-algebraic sets. If $({}^*S_i)_i$ is a t-stratification of $^*X$, we say that $(S_i)_i$ is an archimedean t-stratification of $X$.
\end{Definition} 
In~\cite[Section 7]{immi} the following was proved.

\begin{Fact}\label{inpaper} 
 Let $X$ be a semi-algebraic subset of $\mathbb R^n$ and $(S_i)_i$ be an archimedean t-stratification of $X$. Then $(S_i)_i$ is a Whitney stratification of $X$. 
\end{Fact}

Whitney stratifications are an important tool in singularity theory and they can be traced back to the work of H. Whitney~\cite{whitney}. A definition of them can be found in Section 7 of~\cite{immi}.

Imitating the definition of $(\mathcal C_{p,i})_i$ for a t-stratification $(S_i)_i$, we define the stratification $(C_{p,i})_i$ as follows.
\begin{Definition}\label{classicalstrata}
 For a partition $(S_i)_i$ of $\mathbb R^n$ and fixed $p\in \mathbb R^n$, we define the partition $(C_{p,i})_i$ as $C_{p,0}= C_p(S_0)$
and for $0\leq i<n$, $C_{p,i+1}= C_p(S_{\leq i+1})\setminus C_p(S_{\leq i})$.
\end{Definition}

Our last result, Theorem~\ref{intro2} in the introduction, is an immediate corollary of Theorem~\ref{main} and Fact~\ref{inpaper}.

\begin{Theorem}\label{maincor}
Suppose that $X$ is a semi-algebraic subset of $\mathbb R^n$, $p\in \mathbb R^n$ and $(S_i)_i$ is an archimedean t-stratification of $X$. Then $(C_{p,i})_i$ is an archimedean t-stratification of $C_p(X)$. Particularly, $(C_{p,i})_i$ is a Whitney stratification on $C_p(X)$.   
\end{Theorem}

This result could be phrased as archimedean t-stratifications induce archimedean t-stratifications and, hence, Whitney stratifications on tangent cones.

Our last remark is that the second statement of this theorem marks a contrast between archimedean t-stratifications and Whitney stratifications, as it is known that Whitney stratifications are not strong enough to induce Whitney stratifications on tangent cones, as the following example shows.
\begin{Example}\label{example}
 Consider the set 
\[
X:=\{(x,y,z)\in \mathbb R^3\ |\ x^3-y^2-z^2=0\}.
\] 
The sets $S_0:=\{0\}$, $S_1:=\emptyset$, $S_2:=X\setminus \{0\}$ and $S_3:=\mathbb R^3\setminus X$ constitute a Whitney stratification of $X$. Recall that the tangent cone $C_0(X)$ is the set $\mathbb R_{\geq 0}\times \{0\}\times \{0\}$. Following Definition~\ref{classicalstrata}, we obtain the sets $C_{0,0}=\{0\}$, $C_{0,1}=\emptyset$, $C_{0,2}=\mathbb R_{>0}\times \{0\}\times \{0\}$ and $C_{0,3}=\mathbb R^3\setminus (\mathbb R_{\geq 0}\times \{0\}\times \{0\})$. The dimension of $C_{0,2}$ is 1 and this makes impossible for  $(C_{0,i})_i$ to be a Whitney stratification of $C_0(X)$.

As an example of an archimedean t-stratification, the sets $S_0:=\{0\}$, $S_1:=\mathbb R_{>0}\times \{0\}\times \{0\}$, $S_2:=X\setminus \{0\}$ and $S_3:=\mathbb R^3\setminus (X\cup \mathbb R_{>0}\times \{0\}\times \{0\})$ form an archimedean t-stratification of $X$ (recall that this means that $(^*S_i)_i$ is a t-stratification of $^*X$). In this case we get the sets $C_{0,0}=\{0\}$, $C_{0,1}=\mathbb R_{> 0}\times \{0\}\times \{0\}$, $C_{0,2}=\emptyset$ and $C_{0,3}=\mathbb R^3\setminus (\mathbb R_{\geq 0}\times \{0\}\times \{0\})$, and they do constitute a Whitney stratification of $C_0(X)$ as expected.
 \end{Example}
\end{section}

\begin{section}{Comments on generalising the main results}\label{comments}
In the classical ambit, it is known that the existence of Whitney Stratifications is not restricted to those of semi-algebraic subsets of $\mathbb R^n$. They also exist for more generally definable subsets of $\mathbb R^n$. The main examples of such sets are the sub-analytic sets and the sets definable after adding the exponential map to $\mathcal L_{or}$. Perhaps the most general result on existence of Whitney stratifications is that they exist for any definable subset in any o-minimal expansion of $(\mathbb R,+,\cdot,0,1,<)$ (which includes both of the examples mentioned). This was proved by T. L. Loi in~\cite{loi}. 

This line of thought suggests that our results on stratifications of tangents cones may hold in wider generality. Below, we first discuss a generalisation of Theorem~\ref{main} on induced t-stratifications of tangent cones. We then comment on the correspondent extension of Theorem~\ref{maincor} on archimedean t-stratifications and Whitney stratifications of tangent cones. All the results stated depend on the proof of others (as accordingly discussed).

For the rest of this section, let $\mathcal L$ be a language containing $\mathcal L_{or}$ and $\mathcal T$ an o-minimal $\mathcal L$-theory containing $RCF$ (think of sub-analytic or exponetial-definable setting). The tangent cone of a subset of a model of $\mathcal T$ is defined just as before. When $R$ is non-archimedean and  $V$ is taken as before, t-stratifications can be defined with no further obstacle. One of the first relevant questions in this new setting is the existence of t-stratifications. According to Theorem~\ref{generalexist}, for this to be guaranteed we need $\mathcal T$ to satisfy Hypothesis 2.21 in~\cite{immi}. As a comment on the sort of problem this represents, one needs to show that the structure of $\text{RV}$ is stably embedded (in every model of $\mathcal T$), that $\mathcal T$ has the Jacobian property and that, roughly, t-stratifications of 1-dimensional definable sets exist already. Let us comment that it is informally conjectured that these conditions are satisfied if power boundedness of $\mathcal T$ is assumed.

If the existence of t-stratifications in non-archimedean models of $\mathcal T$ is given, we may then ask whether Theorem~\ref{main} holds. The partition $(\mathcal C_{p,i})_i$ is defined here identically as before for a t-stratification $(S_i)_i$. 

\begin{Theorem}[Assuming t-stratifications exist in models of $\mathcal T$]\label{genmain} Suppose that $\mathcal T$ is \\power bounded. Let $R$ be a non-archimedean model of $\mathcal T$ and suppose that $V$ is a $\mathcal T$-convex subring of $R$. If $X$ is an $\mathcal L\cup \{V\}$-de\-fi\-na\-ble subset of $R^n$ and $(S_i)_i$ is a t-stratification of $X$, then $(\mathcal C_{p,i})_i$ is a t-stratification of $\mathcal C_p(X)$.
\end{Theorem}
\begin{proof}
 The description of $\mathcal C_p(X)$ via definable curves, Proposition~\ref{curvedef2}, is valid in this setting since all the results in Subsection~\ref{curvedefinition} hold for ($R,V$). Actually, this is true regardless of power boundedness, taking into account the generality of the results in Section 2.5 of~\cite{driesII}.  The proof of Theorem~\ref{main} can be then repeated here. The o-minimality of the value group, used at one step of the proof, comes from the power boundedness of $\mathcal T$ (by Proposition 4.3 in~\cite{driesII}). The rest of the proof fits here with only minor modifications.
\end{proof}

Now we discuss the case of archimedean t-stratifications and Whitney stratifications. We assume $\mathcal T$ is power bounded. Considering $\mathbb R$ as an $\mathcal L$-structure, we may wonder whether Theorem~\ref{maincor} holds true for definable sets in this enlarged language. After Theorem~\ref{genmain}, it is obvious that the first part of it holds true, i.e. archimedean t-stratifications induce archimedean t-stratifications on tangent cones. But the second part of this result needs to be proved. 

\vspace{0.1cm}
{\textsc{Conjectured Lemma}}. Let $X$ be an $\mathcal L$-definable subset of $\mathbb R^n$ and suppose that $(S_i)$ is an archimedean t-stratification of $X$. Then $(S_i)_i$ is a Whitney stratification of $X$.

\end{section}

\bibliographystyle{plain}
\bibliography{EGRArxivSub10Sep2015}

\begin{thebibliography}{1}

\bibitem{rag}
J.~Bochnak, M.~Coste, and M.-F. Roy.
\newblock {\em Real Algebraic Geometry. {Ergebnisse der Mathematik und ihrer
  Grenzgebiete, Folge 3}}.
\newblock Springer-Verlag, 1998.

\bibitem{driesII}
L.~{van den} Dries.
\newblock T-convexity and tame extensions {II}.
\newblock {\em The Journal of Symbolic Logic}, 62(1), 1997.

\bibitem{driesI}
L.~{van den} Dries and A.~H. Lewenberg.
\newblock T-convexity and tame extensions.
\newblock {\em The Journal of Symbolic Logic}, 60(1), 1995.

\bibitem{fortuna}
M.~Ferrarotti and E.~Fortuna.
\newblock On semialgebraic tangent cones.
\newblock {\em Bulletin of the Polish Academy of Sciences}, 46, 1998.

\bibitem{wilson}
M.~Ferrarotti, E.~Fortuna, and L.~Wilson.
\newblock Real algebraic varieties with prescribed tangent cones.
\newblock {\em Pacific Journal of Mathematics}, 194(2), 2000.

\bibitem{immi}
I.~Halupczok.
\newblock Non-archimedean {Whitney} stratifications.
\newblock {\em Proceedings of the London Mathematical Society}, 109, 2013.

\bibitem{loi}
T.~L. Loi.
\newblock Verdier and strict {Thom} stratifications in o-minimal structures.
\newblock {\em Illinois Journal of Mathematics}, 42(2), 1988.

\bibitem{whitney}
H.~Whitney.
\newblock Tangents to an analytic variety.
\newblock {\em Annals of Mathematics, Second Series}, 81(3), 1965.

\end{thebibliography}

\end{document}